\documentclass[12pt]{article}

\usepackage[a4paper,
left=2.5cm,
right=2.5cm,
top=2.8cm,
bottom=2.8cm]{geometry}

\usepackage{setspace}
\usepackage{indentfirst}
\usepackage{amsmath, amssymb, amsfonts, amsthm}
\usepackage{mathrsfs}
\usepackage{graphicx}
\usepackage{caption}
\usepackage{float}
\usepackage{booktabs}
\usepackage{array}
\usepackage{tabularx}
\usepackage{multirow}
\usepackage{enumerate}
\usepackage{hyperref}
\usepackage[numbers,sort&compress]{natbib}

\newtheorem{lem}{Lemma}[section]
\newtheorem{thm}[lem]{Theorem}
\newtheorem{cor}[lem]{Corollary}

\newtheorem*{claim*}{Claim}

\usepackage{titlesec}
\usepackage{xcolor}
\newcommand{\setupappendix}{
	\titleformat{\section}[block]
	{\normalfont\Large\bfseries} 
	{Appendix \Alph{section}:}
	{0.5em} 
	{} 
}

\title{An excluded minor theorem for the 6-wheel}

\begin{document}
 	
 	\author{Zijun Chen, Yuqi Xu,\enspace Weihua Yang\footnote{Corresponding author. E-mail: ywh222@163.com,~yangweihua@tyut.edu.cn}\\
 		\small Department of Mathematics, Taiyuan University of Technology,\\
 		\small Taiyuan Shanxi-030024, China}
 	\date{}
 	
 	\maketitle
 	
 	\begin{abstract}
 		For each integer $n \geq 3$, the wheel graph $W_n$ is defined as the graph obtained by connecting a single vertex to all vertices of a cycle of length $n$. In particular, $W_6$ can be uniquely obtained from the Petersen graph by contracting three edges incident to a common vertex. Gubser provided a characterization of all $3$-connected planar $W_6$-minor-free graphs. In this paper, we complete the characterization of $W_6$-minor-free graphs by determining the $3$-connected nonplanar cases.
 	\end{abstract}
 	
 	\noindent\textbf{Keywords:} minor-free; $W_6$; $3$-connected graph; nonplanar
 	
\section{Introduction}
 	
 	All graphs considered in this paper are simple. Let $G$ and $H$ be graphs. A graph $H$ is a minor of $G$, denoted by $H\preceq G$, if it can be obtained from $G$ by deleting and contracting edges. A graph $G$ is said to be $H$-minor-free if no minor of $G$ is isomorphic to $H$. Many fundamental problems in graph theory can be expressed in terms of $H$-minor-free graphs. One of the most well-known conjectures related to the Petersen graph was proposed by Tutte~\cite{tutte1966algebraic}, who stated that every bridgeless graph without a Petersen minor admits a nowhere-zero 4-flow.
 	
 	Although no explicit structure theorems are available for Petersen-minor-free graphs, several results have been established for excluding smaller minors. For several $3$-connected graphs with twelve edges, three classical excluded-minor characterizations are known: Maharry~\cite{maharry2000characterization} characterized cube-minor-free graphs, Ding~\cite{ding2013characterization} characterized octahedron-minor-free graphs, and Maharry and Robertson~\cite{maharry2016structure}
 	characterized $V_8$-minor-free graphs. For each integer $n\geq 3$, let $W_n$ denote the wheel graph obtained by connecting a single vertex to all vertices of a cycle $C_n$.
 	Dirac~\cite{dirac1952property}
 	proved that $W_3$-minor-free graphs are series-parallel graphs. Tutte~\cite{tutte1961theory} showed that $W_3$ is the only $3$-connected $W_4$-minor-free graph.
 	Oxley~\cite{oxley1989regular} characterized $3$-connected $W_5$-minor-free graphs. Gubser~\cite{gubser1993planar} characterized all $3$-connected planar $W_6$-minor-free graphs. The graph $W_6$ also has twelve edges and can be obtained from the Petersen
 	graph by contracting three edges incident with a common vertex. In this paper, we determine all $3$-connected nonplanar $W_6$-minor-free graphs.
 	
 	Let $k\geq 0$ be an integer. A $k$-separation of a graph $G$ is a pair $(G_1,G_2)$ of subgraphs of $G$ such that $E(G_1)\cup E(G_2)=E(G)$, $V(G_1)\cup V(G_2)=V(G)$, $V(G_1)-V(G_2)\neq\emptyset$, $V(G_2)-V(G_1)\neq\emptyset$, and $|V(G_1)\cap V(G_2)|=k$. If $G$ is $3$-connected with at least five vertices, then $G$ is internally $4$-connected if, for every 3-separation $(G_1,G_2)$ of $G$, one of $G_1$ and $G_2$ is isomorphic to $K_{1,3}$.
 	
 	For each integer $n\geq 3$, let $DW^+_n$ denote the double-wheel obtained from a cycle $C_n$ by adding two adjacent vertices $u$ and $v$ and joining each of them to every vertex of the cycle. For each integer $n\geq 3$, an alternating double-wheel $AW^+_{2n}$ is a subgraph of $DW^+_{2n}$ in which $u$ and $v$ are alternately adjacent to every vertex in $C_{2n}$. Moreover, let $DW_n$ and $AW_{2n}$ be the graphs obtained from $DW^+_n$ and $AW^+_{2n}$, respectively, by deleting the edge $uv$. Let $V_8$ be the graph obtained from a cycle $C_8$ by adding an edge between each pair of diametrically opposite vertices. The line graph $L(G)$ of $G$ is the graph whose vertices are the edges of $G$, with two vertices adjacent in $L(G)$ if and only if the corresponding edges of $G$ share an end vertex.
 	For each integer $n\geq 5$, let $C^2_n$ be a graph obtained from a cycle $C_n$ by adding edges between all pairs of vertices of distance two in $C_n$.
 	$K_{m,3}$ is a complete bipartite graph with two disjoint vertex sets, $X_1$ and $X_2$, where $|X_1|=m$ and $|X_2|=3$. Let $K^{i,j}_{m,3}$, with $0\leq i\leq m(m-1)/2$ and $0\leq j\leq 3$, denote the graph obtained by adding $i$ edges between vertices of $X_1$ and $j$ edges between vertices of $X_2$. Let $G+e$ denote a graph obtained from a graph $G$ by adding an edge between two nonadjacent vertices. We write $G+e_1(e_2)$ to indicate that $G+e_1$ is isomorphic to $G+e_2$. Our main results are as follows.
 	
 	\begin{thm}\label{thm1.1}
 		The internally 4-connected $W_6$-minor-free graphs are precisely $\{K_5\}$ $\cup$  $\{K_{3,3}, DW^+_4, K_6\setminus e, K_6\}$ $\cup$ $\{\Gamma_1, K^{2,0}_{4,3}, C^2_7, K^{3,0}_{4,3}, K^{3',0}_{4,3}, C^2_7+e, K^{2',1}_{4,3}, K^{3',1}_{4,3}, K^{4,0}_{4,3}, \Gamma_2, K^{4,1}_{4,3}\}$ $\cup$ $\{C^2_6, DW_5, \mathrm{cube}, C^2_8, N, Q, AW_8\} \cup \{V_8+13+24, V_8+14+58, V_8+13+24+58, V_8+13+14+24+58\}$.

 	\end{thm}
 	
 	Some graphs in Theorem~\ref{thm1.1} are shown in Figures~\ref{fig7}, \ref{fig8} , and~\ref{fig9}. For integers $i,n,k,r\geq 1$, let $G^i$, $G^n$, $G^k$, and $G^r$ be the graphs obtained from $G$ by adding, respectively, $i$ new cubic vertices adjacent to $\{1,3,6\}$, $n$ new cubic vertices adjacent to $\{2,4,6\}$, $k$ new cubic vertices adjacent to $\{3,4,5\}$, and $r$ new cubic vertices adjacent to $\{3,5,y\}$.
 	
 	\begin{thm}\label{thm1.2}
 		The $3$-connected nonplanar $W_6$-minor-free graphs are precisely the $3$-connected nonplanar minors of the following graphs: $V_8+13+16+46$,
 		$V_8+13+14+24+27$, $V_8+13+14+25+27$, $V_8+13+14+36+46$, $V_8+13+14+36+47$, $V_8+13+14+38+46$, $V_8+13+14+58+68$, $V_8+14+16+27+47$, $V_8+13+14+24+25+58$, $V_8+13+14+27+28+47$, $V^i_8+13+16+36$, $K^{3,1}_{4,3}$, $K^{4,1}_{4,3}$, $H_{2a}+16+7z$, $H_{2a}+34+7z$, $H_{2a}+16+34+56$, $H_{2c}+12+34+3y$, $H_{2c}+12+3z+5y$, $\mathrm{cube}^1+3x$, $H_{1b}+12+34+3x$, $H_{2b1}+12+34$, $H^k_{2a1}+34+35+45$, $H_{1a1}$, $\mathrm{cube}^n+24+26+46$, $H^k_{2a2}+34+35+45$, $H^k_{2c1}+34+35+3y+45$, $H^k_{1b1}+34+35+45+5z$, or $H^{k,r}_{2e}+34+35+3y+45+5y$.
 	\end{thm}
 	
 	\begin{figure}[H]
 		\centering
 		\input{last.tpx}
 		\label{fig1}
 	\end{figure}
 	
 	\begin{cor}
 	Let $G$ be a $3$-connected graph. If no two cubic vertices in $G$ share the same neighborhood, and the order of $G$ is strictly greater than $11$, then $G$ contains a $W_6$ minor.
 	\end{cor}
 	
\section{Preliminaries}
	
	Let $G\setminus e$ denote the graph obtained from $G$ by deleting an edge $e$. Conversely, $G$ can be recovered from $G\setminus e$ by adding the edge $e$. We use $G/e$ to denote the graph obtained from $G$ by first contracting an edge $e$, then deleting all but one edge from each parallel family. The reverse operation of contracting an edge is splitting a vertex. Let $N_G(v)$ denote the set of vertices adjacent to $v$. Let $(X,Y)$ be a partition of $N_G(v)$ such that
	$|X|,|Y|\geq 2$. Splitting $v$ produces a graph $G'$ obtained from $G \setminus v$ by adding two new adjacent vertices $x$ and $y$, joining $x$ to all vertices in $X$, and joining $y$ to all vertices in $Y$. The graph $G'$ is called a split of $G$.
	
	Let $G_1$ and $G_2$ be two $3$-connected graphs with cubic vertices $x\in V(G_1)$ and $y\in V(G_2)$. A graph $G$ is called a $T$-sum of $G_1$ with $G_2$ if it is obtained by deleting $x$ and $y$, adding a matching between $N_{G_1}(x)$ and $N_{G_2}(y)$, contracting some of the matching edges, and then simplifying the resulting graph.
	If exactly $i$ matching edges are contracted, then $G$ is called a $T_i$-sum.
	
	For $k=0,1,2$, we define a $k$-sum as follows. Let $G_1$ and $G_2$ be disjoint graphs with more than $k$ vertices. The $0$-sum of $G_1$ with $G_2$ is their disjoint union; a $1$-sum of $G_1$ with $G_2$ is obtained by identifying one vertex of $G_1$ with one vertex of $G_2$; and a $2$-sum of $G_1$ with $G_2$ is obtained by identifying an edge of $G_1$ with an edge of $G_2$, where the common edge may be deleted after the identification.
	
	\begin{lem}[\cite{ding2013characterization}]
	If $H$ is $3$-connected, then the $H$-minor-free graphs are precisely the graphs obtained by taking $0$-, $1$-, and $2$-sums of graphs in the following classes: $K_1$, $K_2$, $K_3$, and $3$-connected $H$-minor-free graphs.
	\end{lem}
	
	Thus, it suffices to characterize the $3$-connected $W_6$-minor-free graphs. The next is a characterization of $V_8$-minor-free graphs.
	
	\begin{lem}[\cite{maharry2016structure}]\label{lemma2.2}
	Every internally $4$-connected $V_8$-minor-free graph $G$ is one of the following: a planar graph, a graph on at most seven vertices, a double-wheel or an alternating double-wheel, a graph with four vertices meeting all edges, or the line graph of $K_{3,3}$.
	\end{lem}
	
	\begin{figure}[H]
		\centering
		\input{V8.tpx}
	\end{figure}
	
	\begin{lem}[\cite{ferguson2015excluding}]\label{lemma2.3}
		A $3$-connected graph is $V_8$-minor-free if and only if it is constructed by repeated $T$-sums of $K_4$ and internally $4$-connected $V_8$-minor-free graphs.
	\end{lem}
	
	We shall divide the proof according to whether a $3$-connected graph contains a $V_8$ minor. Section~3 treats the $V_8$-containing case. Section~4 determines the internally $4$-connected
	$\{V_8,W_6\}$-minor-free graphs. Section~5 extends the analysis to all $3$-connected nonplanar $\{V_8,W_6\}$-minor-free graphs. 

\section{$V_8$-containing graphs}

	In this section, we consider $W_6$-minor-free graphs that contain a $V_8$ minor. A graph is said to be $V_8$-containing if it contains $V_8$ as a minor. The following is a classical result of Seymour~\cite{seymour1980decomposition}, which explains how $3$-connected graphs can be generated.

	\begin{lem}\label{lemma3.1}

	Let $H$ and $G$ be $3$-connected graphs such that $H$ is a proper minor of $G$. Suppose that $H$ is not isomorphic to $W_3$ and that $G$ is not isomorphic to a wheel. Then $G$ has a $3$-connected minor $F$ such that $H\preceq F$, and $F$ is obtained from $H$ by either adding one edge or splitting one vertex.
	\end{lem}

	Splitting a vertex increases the order of the graph by one, whereas adding an edge does not change the order of the graph. The computation proceeds as follows. Starting from the fixed labeling of $V_8$, we consider all subsets of the nonedges of $V_8$. For each resulting graph, we test whether it contains a $W_6$ minor. Isomorphic duplicates are then removed. This produced exactly the 49 graphs listed in Table~\ref{table1}. 

	\begin{lem}\label{lem3.2}
	The only $3$-connected graphs of order $8$ that are $V_8$-containing and $W_6$-minor-free are the $49$ graphs listed in Table~\ref{table1}. 
	\end{lem}

	\begin{table}[H]
	\centering
	\caption{$V_8$-containing and $W_6$-minor-free graphs of order 8} 
	\label{table1}
	\begin{tabular}{ | m{4cm}<{\centering} | m{10cm}<{\centering}  |}
		\hline
		\textbf{Number of Additional Edges}  &  	\textbf{$W_6$-Minor-Free Graphs}  \\ \hline
		$0$ &  $V_8$ \\ \hline
		$1$ &  $V_8+13$, $V_8+14$ \\ \hline
		$2$ &  $V_8+13+14$, $V_8+13+16$, $V_8+13+24$, $V_8+13+25$, $V_8+13+46$, $V_8+13+47$, $V_8+14+16$, $V_8+14+25$, $V_8+14+27$, $V_8+14+58$ \\ \hline
		$3$ & $V_8+13+14+24$, $V_8+13+14+25$, $V_8+13+14+27$, $V_8+13+14+28$, $V_8+13+14+36$, $V_8+13+14+38$, $V_8+13+14+46$, $V_8+13+14+47$, $V_8+13+14+58$, $V_8+13+14+68$, $V_8+13+16+28$, $V_8+13+16+36$, $V_8+13+16+46$, $V_8+13+16+58$, $V_8+13+24+58$, $V_8+13+25+27$, $V_8+13+25+58$, $V_8+13+47+58$, $V_8+14+16+27$, $V_8+14+16+36$\\\hline
		$4$ & $V_8+13+14+24+25$, $V_8+13+14+24+27$, $V_8+13+14+24+58$, $V_8+13+14+25+27$, $V_8+13+14+25+58$, $V_8+13+14+27+28$, $V_8+13+14+27+47$, $V_8+13+14+28+47$, $V_8+13+14+36+46$, $V_8+13+14+36+47$, $V_8+13+14+38+46$, $V_8+13+14+38+47$, $V_8+13+14+58+68$, $V_8+14+16+27+47$ \\ \hline
		$5$ &  $V_8+13+14+24+25+58$, $V_8+13+14+27+28+47$ \\ \hline
		$\geq 6$  &  - \\ \hline
	\end{tabular}
	\end{table}

	Based on Lemma~\ref{lem3.2}, we consider splitting these graphs. For each integer $i\geq 1$, let $V^i_8$ be the graph obtained from $V_8$ by adding $i$ new cubic vertices, each adjacent precisely to $\{1,3,6\}$.
	
	\begin{lem}\label{lem3.3}
	 Every $3$-connected graph of order strictly greater than $8$ that is $V_8$-containing and $W_6$-minor-free is a minor of $V^i_8+13+16+36$ for some integer $i\geq 1$.
	\end{lem}

	\begin{proof}
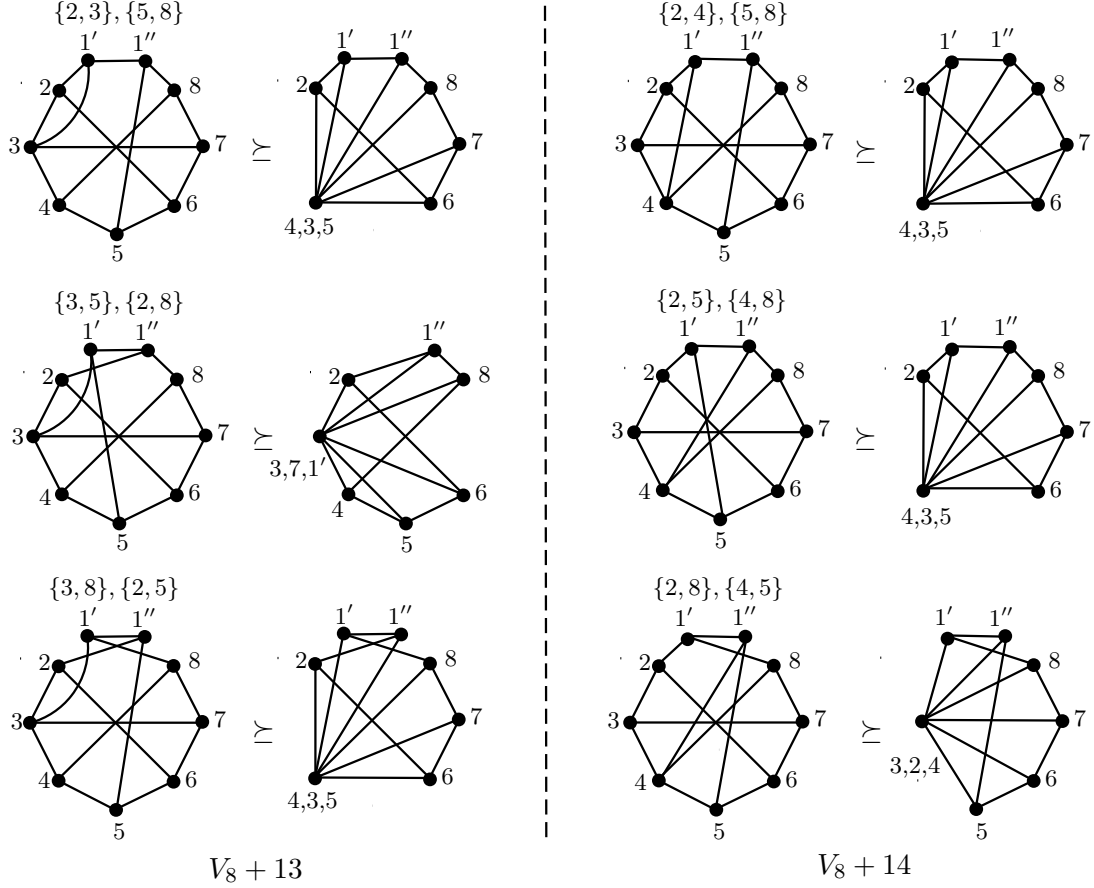
	Note that all vertices of $V_8$ are cubic, so no vertex can be split. Up to symmetry, there are only two possibilities for $V_8+e$, namely $V_8+13$ and $V_8+14$. For each admissible vertex split of these two graphs, the resulting graph contains a $W_6$ minor. The corresponding minor models are shown in Figure~\ref{fig3}.
	
	\begin{figure}[H]
		\input{v.tpx}
		\label{fig3}
	\end{figure}
	
	Next, we consider $V_8+e_1+e_2$, using  $V_8+13+ xy$ as an example. If the edge $xy$ is not incident with the edge $13$, then the split can be regarded as a split performed in either $V_8+13$ or $V_8+14$, followed by the addition of the edge $xy$. Since the split of $V_8+13$ or $V_8+14$ already yields a graph containing a $W_6$ minor, and adding edges preserves the existence of a minor, the resulting graph also contains a $W_6$ minor. Next, we consider $V_8+13+1x$. By symmetry and by the preceding reduction, it remains only to consider the cases in which $1'$ is adjacent to $6$ and one other vertex. The graph obtained after the split remains $W_6$-minor-free only when $N_{V_8+13+14}(1')\setminus\{1''\}=\{3,6\}$. The remaining cases were checked by the same exhaustive procedure. The $W_6$-minor-free outcomes are exactly those listed in
	Table~\ref{table2}.
		
	\begin{table}[H]
	\centering
	\caption{$W_6$-minor-free graphs via order $8\to 9$ vertex splits}
	\label{table2}
		\begin{tabular}{ | c | c|c|c|}\hline 
		\textbf{Original Graph} & \textbf{ Vertex} &	\textbf{Partition} & \textbf{Resulting Graph}  \\ \hline
		
		$V_8+13+16$ & $1$ & $\{3,6\}, \{2,5,8\}$ & $M$ \\ \hline
				
		$V_8+14+16$ & $1$ & $\{4,6\}, \{2,5,8\}$ & $M$ \\ \hline
		
		\multirow{3}{*}{	$V_8+13+16+36$} 
			& $1$ & $\{3,6\}, \{2,5,8\}$ & $M+36$ \\ \cline{2-4}
			& $3$ & $\{1,6\}, \{2,4,7\}$ & $M+36$ \\ \cline{2-4}
			& $6$ & $\{1,3\}, \{2,5,7\}$ & $M+13$ \\ \cline{1-4}
		\end{tabular}
		\end{table}
		
		Note that $M$ with vertex $1'$ deleted is isomorphic to $V_8$. We relabel $1''$ as $1$ and denote the graph $M$ by $V^1_8$. Equivalently, $V^1_8$ is obtained from $V_8$ by adding one new cubic vertex adjacent to $\{1,3,6\}$.
		Adding edges to $V^1_8$ and retaining only the $W_6$-minor-free graphs yields
		$V^1_8$, $V^1_8+13$, $V^1_8+36$, $V^1_8+13+36$, $V^1_8+16+36$, and $V^1_8+13+16+36$. Among the admissible vertex splits of these graphs, the only resulting graphs that remain $W_6$-minor-free are
		$V^2_8$, $V^2_8+13$, $V^2_8+36$, $V^2_8+13+36$, $V^2_8+16+36$, and $V^2_8+13+16+36$.
		
		We now prove by induction on $i$ that, after the first split, every further $W_6$-minor-free admissible split only adds one new cubic vertex adjacent to $\{1,3,6\}$. For $i=1$, this follows from the computation summarized in Table~\ref{table2}. Assume the statement holds for $i$. Since the newly added cubic vertices are pairwise nonadjacent and have the same neighborhood $\{1,3,6\}$, splitting any old vertex reduces, up to the same local configurations already considered for $i=1$, to the cases listed in Table~\ref{table2}. Splitting a newly added cubic vertex is impossible under our definition of vertex splitting, because its neighborhood has size three and cannot be partitioned into two parts of size at least two. Therefore, the only $W_6$-minor-free split produces $V^{i+1}_8$, possibly together with the same edge set among $\{13,16,36\}$.
	
		\begin{figure}[H]
			\input{v2.tpx}
		\end{figure}
		
		\begin{claim*}
			For each $i\geq 1$, the graph $V^i_8+13+16+36$ is $W_6$-minor-free.
		\end{claim*}

		The vertices added in the construction of $V^i_8$ have the same neighborhood $\{1,3,6\}$. A graph of order 6 that contains two vertices of degree two sharing the same neighbors cannot contain a cycle $C_6$. Therefore, we only need to consider the graph $V_8+13+16+36$, which is $W_6$-minor-free by Lemma~\ref{lem3.2}.
	\end{proof}

	We state a corollary of Lemma~\ref{lem3.3} that will be useful in Section 5.

	\begin{lem}\label{lemma3.4}
	Let $G$ be a $V_8$-containing and $W_6$-minor-free graph. If no two cubic vertices share the same  neighborhood, then its order must be $8$.
	\end{lem}

	\begin{lem}\label{lemma3.5}
	The only $3$-connected graphs that are $V_8$-containing and $W_6$-minor-free are minors of one of the following graphs: $V_8+13+16+46$, $V_8+13+14+24+27$, $V_8+13+14+25+27$, $V_8+13+14+36+46$, $V_8+13+14+36+47$, $V_8+13+14+38+46$, $V_8+13+14+58+68$, $V_8+14+16+27+47$, $V_8+13+14+24+25+58$, $V_8+13+14+27+28+47$
	or $V^i_8+13+16+36$ for some integer $i\geq 1$.
	\end{lem}

	For the $3$-connected graphs mentioned in Lemmas~\ref{lem3.2} and~\ref{lem3.3}, we determine whether they are internally $4$-connected.

	\begin{lem}\label{lemma3.6}
	The only internally $4$-connected graphs that are $V_8$-containing and $W_6$-minor-free are $V_8+13+24$, $V_8+14+58$, $V_8+13+24+58$, and $V_8+13+14+24+58$.
	\end{lem}

\section{Internally $4$-connected $V_8$-minor-free graphs}

	Based on Lemma~\ref{lemma2.2}, we first consider internally $4$-connected $\{V_8,W_6\}$-minor-free graphs. The following lemma gives a characterization of $K^{\triangledown}_{3,3}$-minor-free graphs.

	\begin{figure}[H]
	\input{k3,3.tpx}
	\end{figure}

	\begin{lem}[\cite{ding2013excluding}]\label{lemma4.1}
	Every $3$-connected $K^{\triangledown}_{3,3}$-minor-free graph belongs to one of the following three families: planar graphs, graphs with six or fewer vertices, or graphs with three vertices meeting all edges.
	\end{lem}

	\begin{lem}\label{lemma4.2}
	If $v$ is a cubic vertex of an internally $4$-connected graph, then $v$ is not in a triangle.
	\end{lem}

	\begin{lem}\label{lemma4.3}
	$W_6$ is a minor of $L(K_{3,3})$, $AW^+_6$, and $K_{4,4}-3K_2$.
	\end{lem}

	\begin{proof}
	For the three graphs listed above, a $W_6$ minor can be obtained by contracting $\{45,56\}$, $78$, and $27$, respectively.
	\end{proof}

	\begin{figure}[H]
	\input{contain.tpx}
	\end{figure}

	\begin{lem}\label{lemma4.4}
	
	(i) The only internally $4$-connected nonplanar graph of order $5$ is $K_5$.
	
	(ii) The only internally $4$-connected $W_6$-minor-free nonplanar graphs of order $6$ are $K_{3,3}$, $DW^+_4$, $K_6\setminus e$, and $K_6$.
	
	(iii) The only internally $4$-connected $W_6$-minor-free nonplanar graphs of order $7$ are $\Gamma_1$, $K^{2,0}_{4,3}$, $C^2_7$, $K^{3,0}_{4,3}$, $K^{3',0}_{4,3}$, $C^2_7+e$, $K^{2',1}_{4,3}$, $K^{3',1}_{4,3}$, $K^{4,0}_{4,3}$, $\Gamma_2$, and $K^{4,1}_{4,3}$.
	\end{lem}

	\begin{figure}[H]
	\input{MJ1.tpx}
	\label{fig7}
	\end{figure}

	\begin{figure}[H]
	\input{k4,3.tpx}
	\label{fig8}
	\end{figure}

	\begin{proof}
	(i) It is trivial.
	
	(ii) We can construct all $3$-connected nonplanar graphs of order $6$ by adding edges to $K_{3,3}$. Note that $K^{i,j}_{3,3}$ is isomorphic to $K^{j,i}_{3,3}$, where $0 \leq i \leq 3$ and $0 \leq j \leq 3$.
	Among these graphs, only $K^{0,0}_{3,3}$, $K^{2,2}_{3,3}$, $K^{3,2}_{3,3}$, and $K^{3,3}_{3,3}$ are internally $4$-connected by Lemma~\ref{lemma4.2}, which are isomorphic to $K_{3,3}$, $DW^+_4$, $K_6\setminus e$, and $K_6$, respectively.
	
	(iii) If $G$ is $K^{\triangledown}_{3,3}$-minor-free and is nonplanar of order $7$, by the preceding characterization, $G$ has three vertices meeting all edges. It follows that $G$ is one of $K_{4,3}$, $K^{0,1}_{4,3}$, $K^{0,2}_{4,3}$, and $K^{0,3}_{4,3}$. However, none of these graphs is internally $4$-connected by Lemma~\ref{lemma4.2}. Since both $G$ and $K^{\triangledown}_{3,3}$ have seven vertices, this minor is obtained by deleting edges only. Hence $G$ can be obtained from $K^{\triangledown}_{3,3}$ by adding edges. The result is shown in Table~\ref{table3}.
	\end{proof}

	\begin{table}[H]
	\centering
	\caption{Edge additions of $K^{\triangledown}_{3,3}$}
	\label{table3}
	\begin{tabular}{|>{\centering\arraybackslash}m{4cm}|>{\centering\arraybackslash}m{7cm}|>{\centering\arraybackslash}m{2cm}|}
		\hline
		\textbf{Number of Additional Edges} & \textbf{$W_6$-Minor-Free Graphs} & \textbf{Rename} \\ \hline
		$0,1$ & - & - \\ \cline{1-3}
		$2$ & $K^{\triangledown}_{3,3}+15+26$ & $\Gamma_1$ \\ \cline{1-3}
		\multirow{2}{*}{$3$} & $K^{\triangledown}_{3,3}+15+16+23$ & $K^{2,0}_{4,3}$ \\ \cline{2-3}
		& $K^{\triangledown}_{3,3}+16+27+35$ & $C^2_7$ \\ \cline{1-3}
		\multirow{4}{*}{$4$} & $K^{\triangledown}_{3,3}+15+16+23+26$ & $K^{3,0}_{4,3}$ \\ \cline{2-3}
		& $K^{\triangledown}_{3,3}+15+16+23+27$ & $K^{3',0}_{4,3}$ \\ \cline{2-3}
		& $K^{\triangledown}_{3,3}+15+16+27+34$ & $C^2_7+e$ \\ \cline{2-3}
		& $K^{\triangledown}_{3,3}+15+23+34+67$ & $K^{2',1}_{4,3}$ \\ \cline{1-3}
		\multirow{3}{*}{$5$} & $K^{\triangledown}_{3,3}+15+16+23+24+34$ & $K^{3',1}_{4,3}$ \\ \cline{2-3}
		& $K^{\triangledown}_{3,3}+15+16+23+27+67$ & $K^{4,0}_{4,3}$ \\ \cline{2-3}
		& $K^{\triangledown}_{3,3}+15+16+27+34+35$ & $\Gamma_2$ \\ \cline{1-3}
		$6$ & $K^{\triangledown}_{3,3}+15+16+23+27+34+67$ & $K^{4,1}_{4,3}$ \\ \cline{1-3}
		$\geq 7$ & - & - \\ \cline{1-3}
	\end{tabular}
	\end{table}

	\begin{lem}\label{lemma4.5}
	If all edges of $G$ are met by exactly four vertices, the only internally $4$-connected $W_6$-minor-free graph is $\mathrm{cube}$.
	\end{lem}

	\begin{proof}
	Suppose $G$ has a set $X$ of four vertices that meet all edges of $G$.
	Let $Y=V(G)-X$ and let $Y_3$, $Y_4$ consist of vertices of $Y$ of degree $3$ and $4$, respectively. We claim that no two vertices in $Y_3$ have the same neighbors. Otherwise, these three neighbors form a $3$-separation, which contradicts the definition of an internally $4$-connected graph. By Lemma~\ref{lemma4.3}, deleting three nonincident edges from $K_{4,4}$ results in a graph $K_{4,4}-3K_2$ that contains $W_6$. Since $|Y| \geq 4$, we have $|Y_4|=0$. By Lemma~\ref{lemma4.2}, it follows that $G$ is isomorphic to the cube.
	\end{proof}

	Now, we consider internally $4$-connected $W_6$-minor-free planar graphs. Gubser characterized $3$-connected planar $W_6$-minor-free graphs. However, it is not straightforward to determine from his characterization which of these graphs are internally $4$-connected. Using this result as a foundation, we can derive an upper bound.

	\begin{lem}[\cite{gubser1993planar}]\label{lemma4.6}
	If $G$ is a $3$-connected $W_6$-minor-free planar graph, then the number of edges of $G$ is less than $18$, and the order of $G$ is less than $12$.
	\end{lem}

	We use Plantri~\cite{397541} to generate all $3$-connected planar graphs satisfying the bounds in Lemma~\ref{lemma4.6}. Then, we determine which of them are internally $4$-connected and $W_6$-minor-free.

	\begin{lem}\label{lemma4.7}
	The only internally $4$-connected $W_6$-minor-free planar graphs are $C^2_6$, $DW_5$, $C^2_8$, $\mathrm{cube}$, $N$, $Q$, and $AW_8$.
	\end{lem}
	
	\begin{figure}[H]
		\input{NQ.tpx}
		\label{fig9}
	\end{figure}

	\begin{proof}
	Using Plantri, we obtain the graphs listed in Table~\ref{table4}.
	
	\begin{table}[H]
		\centering
		\caption{Internally $4$-connected planar graphs}
		\label{table4}
		\begin{tabular}{|c|c|c|c|c|c|c|}
			\hline
			\textbf{Order} & $6$ & $7$ & $8$ & $9$ & $10$ & $11$ \\ \hline
			\textbf{Graphs} & $C^2_6$ & $DW_5$ & $\mathrm{cube}$, $C^2_8$, $C^2_8+e$ & $N$ & $Q$, $AW_8$ & $R$ \\ \hline
		\end{tabular}
	\end{table}
	
	First, Figure~\ref{fig10} shows that $C^2_8+e$ and $R$ contain a $W_6$ minor.
	
	\begin{figure}[H]
		\input{c28.tpx}
		\label{fig10}
	\end{figure}
	
	Now, we prove that the other graphs $\mathrm{cube}$, $C^2_8$, $N$, $Q$, and $AW_8$ are $W_6$-minor-free.

	For $\mathrm{cube}$ and $C^2_8$, any $W_6$ minor would have to be obtained by contracting one edge. However, after contracting any edge, the resulting graph has no vertex of degree $6$. For $N$, no vertex of degree $6$ is produced by contracting two edges.
	
	Since $Q$ has ten vertices and fifteen edges, while $W_6$ has seven vertices and twelve edges, any $W_6$ minor of $Q$ must be obtained by contracting three edges, and these contractions cannot create parallel edges. Otherwise, after suppressing parallel edges, the resulting simple graph would have fewer than twelve edges and hence could not be $W_6$. Therefore, it suffices to consider contractions that create a vertex of degree $6$. Up to symmetry, this can only occur by first contracting the edge $28$ and then contracting either $78$ or $89$.  In each of the two resulting graphs, all edges of the potential degree-$6$ vertex $8$ are contained within triangles.
	
	For $AW_8$, up to isomorphism, the only contractions that can create a vertex of degree $6$ are obtained by contracting one of the edge sets $\{13,14,17\}$, $\{13,17,24\}$, and $\{13,14,12\}$. In each case, the resulting degree-$6$ vertex is $1$. However, in none of these cases does the resulting graph contain a cycle $C_6$ after the deletion of vertex $1$. 
	\end{proof}

	Based on Lemmas~\ref{lemma4.3}--\ref{lemma4.7}, we obtain all internally $4$-connected $\{V_8,W_6\}$-minor-free graphs.

	\begin{lem}\label{lemma4.8}
	The only internally $4$-connected $\{V_8,W_6\}$-minor-free graphs are
	$K_5$, $K_{3,3}$, $DW^+_4$, $C^2_6$, $K_6\setminus e$, $K_6$, $\Gamma_1$, $K^{2,0}_{4,3}$, $C^2_7$, $DW_5$, $K^{3,0}_{4,3}$, $K^{3',0}_{4,3}$, $C^2_7+e$, $K^{2',1}_{4,3}$, $K^{3',1}_{4,3}$, $K^{4,0}_{4,3}$, $\Gamma_2$, $K^{4,1}_{4,3}$, $\mathrm{cube}$, $C^2_8$, $N$, $Q$, and $AW_8$.
	\end{lem}

	\begin{proof}
		Since $DW^+_6$ contains an $AW^+_6$ minor, any double-wheel or alternating double-wheel appearing in the present class must have sufficiently small order. Hence the only such possibilities are $DW^+_3$ and $DW^+_4$. Observe that $DW^+_3$ is isomorphic to $K_5$. The remaining four cases in Lemma~\ref{lemma2.2} are handled by Lemmas~\ref{lemma4.3}--\ref{lemma4.7}.
	\end{proof}

\section{$T$-sum}

	In this section, we determine the $3$-connected nonplanar $\{V_8,W_6\}$-minor-free graphs by applying Lemma~\ref{lemma2.3}.

	Note that $K_5$, $DW^+_4$, $C^2_6$, $K_6\setminus e$, $K_6$, $C^2_7$, $DW_5$, $K^{3',0}_{4,3}$, $C^2_7+e$,  $K^{2',1}_{4,3}$, $K^{3',1}_{4,3}$, $K^{4,0}_{4,3}$, $\Gamma_2$,  $K^{4,1}_{4,3}$ and $C^2_8$ do not participate in any $T$-sums as these graphs do not contain a cubic vertex. This leads to the following characterization of $3$-connected nonplanar $\{V_8,W_6\}$-minor-free graphs.

	\begin{lem}\label{lemma5.1}
	A $3$-connected graph $G$ is $\{V_8,W_6\}$-minor-free if and only if  $G$ belongs to \{$K_5$, $K_{3,3}$, $DW^+_4$, $C^2_6$, $K_6\setminus e$, $K_6$, $\Gamma_1$, $K^{2,0}_{4,3}$, $C^2_7$, $DW_5$, $K^{3,0}_{4,3}$, $K^{3',0}_{4,3}$, $C^2_7+e$, $K^{2',1}_{4,3}$, $K^{3',1}_{4,3}$, $K^{4,0}_{4,3}$, $\Gamma_2$, $K^{4,1}_{4,3}$, $\mathrm{cube}$, $C^2_8$, $N$, $Q$, $AW_8$\} or it is constructible from $K_4$, $K_{3,3}$, $\Gamma_1$, $K^{2,0}_{4,3}$, $K^{3,0}_{4,3}$, $\mathrm{cube}$, $N$, $Q$, and $AW_8$ by repeated $T$-sums.
	\end{lem}

	\begin{lem}
	Let $G_1$ contain a $3$-connected graph $G$ as a minor, and $G_2$ be a $3$-connected graph. If $H$ is a $T$-sum of $G_1$ with $G_2$, then $H$ contains $G$ as a minor. 
	\end{lem}

	\begin{proof}
	If $|G_2|=4$, then $G_2$ is isomorphic to $K_4$. Based on the definition of $T$-sum, at most two matching edges are contracted. We obtain the graph $G_1+e$. Then $H$ contains $G$ as a minor. Now we consider that $|G_2|\geq 5$. Suppose that $u_i$ is a cubic vertex of $G_i$ with neighbors $x_i, y_i, z_i, i=1,2$. Let $H$ be a $T(u_1,u_2)$-sum of $G_1$ with $G_2$. Let $w$ be a vertex of $G_2$. Since $G_2$ is 3-connected, $y_2$ is adjacent to $w$ by symmetry and there are at least three internally disjoint paths between $x_2$ and $w$ by Menger's theorem~\cite{menger1927allgemeinen}. There is an ($x_2, w$)-path that does not contain $y_2$ or $z_2$. Similarly, there is a ($z_2, w$)-path that does not contain $x_2$ or $y_2$. Then $w$ is adjacent to $x_2$, $y_2$, and $z_2$ by contracting some vertices in the paths. $w$ is adjacent to $x_1$, $y_1$, and $z_1$ even if all three matching edges $x_iy_i$ are contracted. Thus, $H$ contains $G$ as a minor.
	\end{proof}

	We only need to consider the result of the $T$-sum of two graphs, both of which are $W_6$-minor-free. Let $\mathcal{H}$ denote all nonplanar graphs that can be constructed from $K_4$, $K_{3,3}$, $\Gamma_1$, $K^{2,0}_{4,3}$, $K^{3,0}_{4,3}$, $\mathrm{cube}$, $N$, $Q$, and $AW_8$ by repeated $T$-sums. Moreover, a $T$-sum of two planar graphs is again planar. Since the planar case has already been characterized by Gubser, no further discussion of this case is needed here.

	\begin{lem}\label{lemma5.3}
	Let $H \in \mathcal{H}$. $H$ is $K^{2,1}_{4,3}$ or $K^{3,1}_{4,3}$ if $K^{2,0}_{4,3}$ or $K^{3,0}_{4,3}$ participates in the $T$-sum. 
	\end{lem}

	\begin{figure}[H]
	\input{kkk43.tpx}
	\end{figure}

	\begin{proof}
	We label the unique cubic vertex as 1. The result of $T_3$-sum of $K^{2,0}_{4,3}$ with $K_{3,3}$ is $K^{2,0}_{5,3}$. By contracting one of the new edges, $K^{2,2}_{4,3}\preceq K^{2,0}_{5,3}$. Since $K^{2,2}_{4,3}$ contains $W_6$	as a minor, it follows that $W_6\preceq K^{2,0}_{5,3}$.  
	It is easy to check that any $T_3$-sum of $K^{2,0}_{4,3}$ with any graph from the set $\{\mathrm{cube}, N, Q, AW_8, K^{2,0}_{4,3}, \Gamma_1, K^{3,0}_{4,3}\}$ contains the graph $J$. By contracting the heavy edge of $J$, we obtain $W_6\preceq J$. Thus, we only need to consider the $T$-sum of $K^{2,0}_{4,3}$ with $K_4$. The result of $T_1$-sum of $K^{2,0}_{4,3}$ and $K_4$ is $J$, which contains a $W_6$ minor. The result of $T_2$-sum of $K^{2,0}_{4,3}$ with $K_4$ is $K^{2,0}_{4,3}+56(57,67)$, which is $K^{2,1}_{4,3}$. We then only need to consider the $T$-sum of $K^{2,1}_{4,3}$ with $K_4$.
		\begin{figure}[H]
		\input{J.tpx}
		\end{figure}
	The $T_2$-sum of $K^{2,1}_{4,3}$ with $K_4$ is $K^{2,2}_{4,3}$ or itself. Since $K^{2,2}_{4,3}$ contains $W_6$, we terminate the discussion regarding the $T$-sum of $K^{2,0}_{4,3}$. $H$ is $K^{2,1}_{4,3}$ if $K^{2,0}_{4,3}$ participates in the $T$-sum.
	
	Adding an edge between the two degree-four vertices in graph $K^{2,0}_{4,3}$ results in graph $K^{3,0}_{4,3}$. The same argument applies to $K^{3,0}_{4,3}$, since the added edge is not incident with the unique cubic vertex and hence does not affect the preceding $T$-sum analysis. Thus, $H$ is $K^{3,1}_{4,3}$ if $K^{3,0}_{4,3}$ participates in the $T$-sum.
	\end{proof}

	\begin{lem}\label{lemma5.4}
	Let $H \in \mathcal{H}$. $\mathcal{H}$ is empty if $N$, $Q$, or $AW_8$ participates in the $T$-sum.
	\end{lem}

	\begin{proof}
	We only need to consider the $T$-sum of these graphs with nonplanar graphs. Note that all cubic vertices of $N$ are structurally equivalent. We take the vertex $4$ as an example. As illustrated in Figure~\ref{fig13}, $W_6$ is a minor of the $T_3$-sum of $N$ with $K_{3,3}$. Thus, $\mathcal{H}$ is empty if $N$ participates in the $T$-sum. 
	
	All cubic vertices of $Q$ are structurally equivalent, as are those of $AW_8$. The same argument applies to $Q$ and $AW_8$: in each case, a $T_3$-sum with $K_{3,3}$ contains a $W_6$ minor. Thus, $\mathcal{H}$ is empty if $Q$ or $AW_8$ participates in the $T$-sum.
	\end{proof}

	\begin{figure}[H]
	\input{N1.tpx}
	\label{fig13}
	\end{figure}

	\begin{lem}
	Let $H \in \mathcal{H}$. $H$ is a minor of $K^{3',1}_{4,3}$ if $\Gamma_1$ participates in the $T$-sum. 
	\end{lem}

	\begin{proof}
	The two cubic vertices $\{3,7\}$ of $\Gamma_1$ are structurally equivalent. The $T_3$-sums of $\Gamma_1$ with $K_{3,3}$ and with $\mathrm{cube}$ each contain a $W_6$ minor. Moreover, the $T_1$-sum of $\Gamma_1$ with $K_4$ contains a $W_6$ minor.
	
	Hence only the $T_2$-sums of $\Gamma_1$ with $K_4$ need to be considered. Each such $T_2$-sum is isomorphic to $\Gamma_1+e$ for some $e\in \{16,17,34,35,45,67\}$. Repeating the same operation only adds further edges from this set. Up to isomorphism, the possible graphs obtained in this way are $\Gamma_1+34$, $\Gamma_1+45$, $\Gamma_1+34+35$, $\Gamma_1+34+16$, $\Gamma_1+34+17$, $\Gamma_1+16+45$, $\Gamma_1+34+35+16$, $\Gamma_1+34+35+17$, and $\Gamma_1+34+35+17+67$. The last graph is isomorphic to $K^{3',1}_{4,3}$. $H$ is a minor of $K^{3',1}_{4,3}$ if $\Gamma_1$ participates in the $T$-sum. 
	\end{proof}

	For each integer $n\geq 1$, let $\mathrm{cube}^n$ be the graph obtained from $\mathrm{cube}$ by adding $n$ new cubic vertices, each adjacent precisely to $\{2,4,6\}$.

	\begin{lem}\label{lemma5.6}
	 Let $H \in \mathcal{H}$. $H$ is a minor of $\mathrm{cube}^1+3x$ or $\mathrm{cube}^n+24+26+46$ if $\mathrm{cube}$ participates in the $T$-sum. 		
	\end{lem}
	
	\begin{figure}[H]
	\input{cube_sum.tpx}
	\end{figure}

	\begin{proof}
	All cubic vertices of $\mathrm{cube}$ are structurally equivalent. The $T_3$-sum of $\mathrm{cube}$ with itself is a planar graph with $18$ edges; by Lemma~\ref{lemma4.6}, it contains a $W_6$ minor. Thus no new $W_6$-minor-free graph arises from this case. 
	
	We next consider $T$-sums of $\mathrm{cube}$ with $K_4$ or $K_{3,3}$. Since only nonplanar graphs are relevant here, a nonplanar summand must occur; hence it suffices to consider the cases involving $K_{3,3}$. The essential cases are listed in Table~\ref{table5}. In the table, the notation $3(1,2,4,5,6,7,8)$ means that vertex $3$ is structurally equivalent to vertices $1,2,4,5,6,7,8$. Hence it is enough to use vertex $3$ as a representative.
	
	\begin{table}[H]
		\centering
		\caption{$T$-sum of $\mathrm{cube}$}
		\label{table5}
		\begin{tabular}{|c|c|c|c|}
			\hline
			\textbf{Graphs} & \textbf{Vertex} & \textbf{Number of Contracted Edges} & \textbf{Result} \\ \hline
			\multirow{2}{*}{{$\mathrm{cube}$ with $K_{3,3}$} } & \multirow{2}{*}{3(1,2,4,5,6,7,8) } & $0,1,2$ & $\succeq W_6$ \\ \cline{3-4}
			&  & 3 & $\mathrm{cube}^1$ \\ \hline
		\end{tabular}
	\end{table}
	
	Because the sequence of $T$-sums involving $K_{3,3}$ may already have advanced, some edges cannot be produced at this stage solely by a $T_2$-sum of $\mathrm{cube}^1$ with $K_4$. We therefore consider all possible edge additions and exclude those that contain a $V_8$ minor. Up to isomorphism, the possible graphs obtained in this way are: $\mathrm{cube}^1+3x$, $\mathrm{cube}^1+24$, $\mathrm{cube}^1+24+26$, and $\mathrm{cube}^1+24+26+46$. The last three graphs arise directly from $T$-sums and are $V_8$-minor-free. The graph $\mathrm{cube}^1+3x$ is also $V_8$-minor-free by Lemma~\ref{lemma3.4}. Indeed, a $T_2$-sum of $K_4$ with $K_{3,3}$ gives $K^{1,0}_{3,3}$, and a subsequent $T_3$-sum of $K^{1,0}_{3,3}$ with $\mathrm{cube}$ gives $\mathrm{cube}^1+3x$.
	
	\begin{table}[H]
		\centering
		\caption{$T$-sum of $\mathrm{cube}^1$}
		\begin{tabular}{|c|c|c|c|}
			\hline
			\textbf{Graphs} & \textbf{Vertex} & \textbf{Number of Contracted Edges} & \textbf{Result} \\ \hline
			
			\multirow{5}{*}{$\mathrm{cube}^1$ with $K_4$} & \multirow{2}{*}{1(5,7)} & $0,1$ & $\succeq W_6$  \\ \cline{3-4}
			&  & 2 & $\mathrm{cube}^1+24$\\ \cline{2-4}
			&\multirow{1}{*}{8 } & 0,1,2 & $\succeq W_6$  \\ \cline{2-4}
			
			&\multirow{2}{*}{3($x$) } & $0,1$ & $\succeq W_6$  \\ \cline{3-4}
			&  & 2 & $\mathrm{cube}^1+24(26,46)$ \\ \cline{2-4}	\hline	
			
			\multirow{4}{*}{$\mathrm{cube}^1$ with $K_{3,3}$}
			&\multirow{1}{*}{1(5,7) } & $0,1,2,3$& $\succeq W_6$  \\ \cline{2-4}
			&\multirow{1}{*}{8} & $0,1,2,3$ & $\succeq W_6$  \\ \cline{2-4}
			
			&\multirow{2}{*}{3($x$) } & $0,1,2$ & $\succeq W_6$  \\ \cline{3-4}
			&  & 3 & $\mathrm{cube}^2$ \\ \cline{2-4}	    \hline
			
		\end{tabular}
	\end{table}
	
	Similarly, we obtain graphs of order 10: $\mathrm{cube}^2+24$, $\mathrm{cube}^2+24+26$, and $\mathrm{cube}^2+24+26+46$. Continuing inductively, all possible graphs in this branch are $\mathrm{cube}^n+24$, $\mathrm{cube}^n+24+26$, and $\mathrm{cube}^n+24+26+46$. Among these, the maximal graph is $\mathrm{cube}^n+24+26+46$. To prove that it is $W_6$-minor-free, it suffices, by the reduction used in Lemma~\ref{lem3.3}, to check the core graph $\mathrm{cube}+24+26+46$. Contracting any edge of this core graph does not produce a vertex of degree $6$; hence the core graph, and therefore $\mathrm{cube}^n+24+26+46$, is $W_6$-minor-free.
	\end{proof}

	Only the cases involving $K_4$ and $K_{3,3}$ remain.
	
	\begin{lem}
	Up to isomorphism, the graphs obtained by one $T$-sum involving $K_4$ with $K_{3,3}$ are $\mathrm{Prism}$, $W_4$, $K_4$, $H_1$, $H_2$, $H_3$, $K^{1,0}_{3,3}$, and $K_{4,3}$.
	\end{lem}

	\begin{figure}[H]
	\input{first.tpx}
	\end{figure}

	\begin{proof}
	The enumeration is given in Table~\ref{table7}.
	\end{proof}

	\begin{table}[H]
	\centering
	\caption{$T$-sum of $K_4$ with $K_{3,3}$}
	\label{table7}
	\begin{tabular}{|c|c|c|}
		\hline
		\textbf{Graphs}  & \textbf{Number of Contracted Edges} & \textbf{Result} \\ \hline
		
		\multirow{3}{*}{$K_4$ with $K_4$} & $0$ & $\mathrm{Prism}$ \\ \cline{2-3}
		&1 & $W_4$ \\ \cline{2-3}
		&2 &  $K_4$  \\ \hline

		\multirow{3}{*}{$K_{3,3}$ with $K_4$} & $0$ & $H_1$ \\ \cline{2-3}
		&1 &  $H_2$ \\ \cline{2-3}
		& 2&  $K^{1,0}_{3,3}$ \\ \hline
		
		\multirow{3}{*}{$K_{3,3}$ with $K_{3,3}$} & $0,1$ & $\succeq W_6$ \\\cline{2-3}
		&  2 & $H_3$ \\ \cline{2-3}
		&  3 & $K_{4,3}$ \\ \hline
	\end{tabular}
	
\end{table}

	Since we are interested only in nonplanar graphs, some step in the $T$-sum sequence must involve a nonplanar graph. Thus it remains to consider the branches arising from $H_1$, $H_2$, $H_3$, $K^{1,0}_{3,3}$, and $K_{4,3}$.

	For integers $k,j\geq 1$, let $G^k$ and $G^j$ be the graphs obtained from $G$ by adding, respectively, $k$ new cubic vertices adjacent to $\{3,4,5\}$ and $j$ new cubic vertices adjacent to $\{3,5,7\}$.

	\begin{lem}\label{lemma5.8}
	Let $H \in \mathcal{H}$. When $H_3$ participates in the $T$-sum, $H$ can be any of the following variants: $H_3$, $H_{3a}$, $H_{3b}$, $H^k_{3}$, $H^{k,j}_{3}$, $H^k_{3a}$, $H^{j}_{3a}$, $H^{k,j}_{3a}$, and $H^k_{3b}$, possibly with additional edges among the existing vertices, subject to the condition that the resulting graph remains $W_6$-minor-free.
	\end{lem}

	\begin{figure}[H]
	\input{h3a.tpx}
	\end{figure}

	\begin{proof}
	By symmetry, the pairs ${1,2}$, ${6,8}$, and ${4,7}$ in $H_3$ are structurally equivalent. When $T$-sums involving $H_3$ are considered, it is enough to use $K_4$ and $K_{3,3}$ as the other summand. This is because all higher-order cases are obtained by repeated $T$-sums with these two graphs. The first round of $T$-sums produces the higher-order graphs $H_{3a}$, $H_{3b}$, and $H^1_3$. Additionally, we also obtain $H_3+e$. Here $e$ is not restricted in advance; instead, all possible edge additions are considered, and those containing a $V_8$ minor are excluded.
	
	\begin{table}[H]
		\centering
		\caption{$T$-sum of $H_3$}
		\begin{tabular}{|c|c|c|c|}
			\hline
			\textbf{Graphs} & \textbf{Vertex} & \textbf{Number of Contracted Edges} & \textbf{Result} \\ \hline
			
			\multirow{9}{*}{$H_3$ with $K_4$} & \multirow{3}{*}{$2(1)$ } & $0$ & $\succeq W_6$  \\ \cline{3-4}
			&  & $1$ & $H_{3a}$\\ \cline{3-4} &  & $2$ & $H_3+34(45), H_3+35$ \\ \cline{2-4}
			
			&\multirow{3}{*}{7(4) } & $0$ & $\succeq W_6$  \\ \cline{3-4}
			&	& $1$ & $H_{3b}$  \\ \cline{3-4}
			&  & $2$ & $H_3+35$, $H_3+35(37)$\\ \cline{2-4}
			&\multirow{3}{*}{$6(8)$ } & $0$ & $\succeq W_6$  \\ \cline{3-4}
			&  & $1$ & $H_{3a}$\\ \cline{3-4}	&  & $2$ & $H_3+34(45), H_3+35$ \\ \cline{1-4}
			
			\multirow{5}{*}{$H_3$ with  $K_{3,3}$} & \multirow{2}{*}{$2(1)$ } & $0,1,2$ & $\succeq W_6$  \\ \cline{3-4}
			&  & $3$ & $H^1_3$ \\ \cline{2-4}
			&\multirow{1}{*}{7(4)} & $0,1,2,3$& $\succeq W_6$  \\ \cline{2-4}
			&\multirow{2}{*}{$6(8)$ } & $0,1,2$ & $\succeq W_6$  \\ \cline{3-4}
			&  & $3$ & $H^{1'}_3$ \\ \cline{1-4}
			
		\end{tabular}
	\end{table}
	
	We now consider graphs of the form $H_3+e$. Since $H_3$ is $3$-connected, adding an edge can only reduce the number of essentially distinct vertices that need to be considered in subsequent $T$-sums. Thus every graph obtained from a further $T$-sum of $H_3+e$ is contained in one of $H_{3a}+e$, $H_{3b}+e$, $H^1_3+e$, $H^{1'}_3+e$, $H_3+e$, and $H_3+e'+e''$.
	
	Similarly, a $T$-sum involving $H_{3a}$ produces only $H^1_{3a}$, $H^{1'}_{3a}$, and $H_{3a}+e$. 
	
	\begin{table}[H]
		\centering
		\caption{$T$-sum of $H_{3a}$}
		\begin{tabular}{|c|c|c|c|}
			\hline
			\textbf{Graphs} & \textbf{Vertex} & \textbf{Number of Contracted Edges} & \textbf{Result} \\ \hline
			
			\multirow{7}{*}{$H_{3a}$ with $K_4$} & \multirow{2}{*}{ $1$ }
			& $0,1$ & $\succeq W_6$ \\ \cline{3-4}
			& & $2$ & $H_{3a}+34(45), H_{3a}+35$  \\  \cline{2-4}
			& \multirow{1}{*}{ $7$ }
			& $0,1,2$ & $\succeq W_6$ \\ \cline{2-4}
			& \multirow{2}{*}{ $x(z)$ }
			& $0,1$ & $\succeq W_6$ \\ \cline{3-4}
			& & $2$ & $H_{3a}+34$  \\  \cline{2-4}
			& \multirow{2}{*}{ $6(8)$ }
			& $0,1$ & $\succeq W_6$ \\ \cline{3-4}
			& & $2$ & $H_{3a}+35$,$H_{3a}+37$,$H_{3a}+57$  \\  \cline{1-4}
			
			\multirow{6}{*}{$H_{3a}$ with $K_{3,3}$ } & \multirow{2}{*}{ $1$ }
			& $0,1,2$ & $\succeq W_6$ \\ \cline{3-4}
			&	& $3$ &  $H^1_{3a}$   \\ \cline{2-4}
			& \multirow{1}{*}{ $7$ }
			& $0,1,2,3$ & $\succeq W_6$ \\ \cline{2-4}
			& \multirow{1}{*}{ $x(z)$ }
			& $0,1,2,3$ & $\succeq W_6$ \\ \cline{2-4}
			& \multirow{2}{*}{ $6(8)$ }
			& $0,1,2$ & $\succeq W_6$ \\ \cline{3-4}
			& & $3$ & $H^{1'}_{3a}$\\  \cline{1-4}
		\end{tabular}
	\end{table}
	
	And a $T$-sum involving $H_{3b}$ produces only $H^1_{3b}$ and $H_{3b}+e$.
	
	\begin{table}[H]
		\centering
		\caption{$T$-sum of $H_{3b}$}
		\begin{tabular}{|c|c|c|c|}
			\hline
			\textbf{Graphs} & \textbf{Vertex} & \textbf{Number of Contracted Edges} & \textbf{Result} \\ \hline
			
			\multirow{4}{*}{$H_{3b}$ with $K_4$} & \multirow{2}{*}{ $1(2)$ }
			& $0,1$ & $\succeq W_6$ \\ \cline{3-4}
			& & $2$ & $H_{3b}+34(45), H_{3b}+35$  \\  \cline{2-4}
			& \multirow{1}{*}{ $6(8)$ }
			& $0,1,2$ & $\succeq W_6$ \\ \cline{2-4}
			& \multirow{1}{*}{ $x(z)$ }
			& $0,1,2$ & $\succeq W_6$ \\ \cline{1-4}
			
			\multirow{4}{*}{$H_{3b}$ with $K_{3,3}$ } & \multirow{2}{*}{ $1(2)$ }
			& $0,1,2$ & $\succeq W_6$ \\ \cline{3-4}
			&	& $3$ &  $H^1_{3b}$   \\ \cline{2-4}
			& \multirow{1}{*}{ $6(8)$  }
			& $0,1,2,3$ & $\succeq W_6$ \\ \cline{2-4}
			& \multirow{1}{*}{ $x(z)$ }
			& $0,1,2,3$ & $\succeq W_6$ \\ \cline{1-4}
			
		\end{tabular}
	\end{table}

	The case $H_{3a}+e$ is analogous to that of $H_3+e$. Finally, applying the same analysis to $H^1_3$, $H^1_{3a}$, $H^{1'}_{3a}$, and $H^1{3b}$ gives precisely the families $H^k_3$, $H^j_3$, $H^{k,j}_3$, $H^k_{3a}$, $H^j_{3a}$, $H^{k,j}_{3a}$, and $H^k_{3b}$, together with their admissible edge additions.
	\end{proof}

	By Lemma~\ref{lemma5.8}, it remains to consider repeated edge additions to $H_3$, $H_3$, $H_{3a}$, $H_{3b}$, $H^k_{3}$, $H^{k,j}_{3}$, $H^k_{3a}$, $H^{j}_{3a}$, $H^{k,j}_{3a}$, and $H^k_{3b}$. Then, we determine which of these graphs are $W_6$-minor-free and identify the maximal graphs in each category.

	\begin{lem}
	All the following graphs are $W_6$-minor-free:
	
	(i) $H_3+34+35+37+45+57$, $H_3+34+35+45+46+68$, $H_3+12+35+68$, $H^{k}_3+34+35+37+45+57$, $H^{k}_3+34+35+45+46+68$,  $H^{k,j}_3+34+35+37+45+57$;
	
	(ii) $H_{3a}+34+35+37+45+57$, $H^{k}_{3a}+34+35+37+45+57$,  $H^{j}_{3a}+34+35+37+45+57$, $H^{k,j}_{3a}+34+35+37+45+57$;
	
	(iii) $H_{3b}+34+35+45+68$, $H^j_{3b}+34+35+45+68$.
	
	\end{lem}

	\begin{proof}
	It suffices to check the following graphs:	$H_3\setminus \{v_1, v_6\}+34+35+37+45+57$, $H_3\setminus \{v_1, v_6\}+34+35+45+46+68$, $H_3+12+35+68$, $H_{3a}\setminus v_6+34+35+37+45+57$, $H_{3b}\setminus v_1+34+35+45+68$. All of them are $W_6$-minor-free. 
	
	We must also exclude the possibility that the relevant edge additions create a $V_8$ minor, as required by Lemma~\ref{lemma3.4}. For example, $H_3+16$ is isomorphic to $V_8+13+46$. The graphs listed above have been checked to be $V_8$-minor-free.
	\end{proof}

	\begin{lem}\label{lemma5.10}
	Let $H \in \mathcal{H}$. $H$ is a minor of one of $H_3+12+35+68$, $H^{k,j}_{3a}+34+35+37+45+57$, or $H^{j}_{3b}+34+35+45+68$ if $H_3$ participates in the $T$-sum.		
	\end{lem}

	\begin{proof}
	Note that  $H^{k,j}_{3}+34+35+37+45+57\preceq H^{k,j}_{3a}+34+35+37+45+57$, $H^{k}_3+34+35+45+46+68\preceq H^{j}_{3b}+34+35+45+68$.
	\end{proof}

	The analyses for $H_1$, $H_2$, $K^{1,0}_{3,3}$, and $K_{4,3}$ are similar to that for $H_3$. We therefore state the conclusions here; the corresponding case checks are given in the Appendix.

	For each integer $l\geq 1$, let $G^l$ be the graph obtained from $G$ by adding $l$ new cubic vertices, each adjacent precisely to $\{1,2,6\}$.
	
	\begin{lem}
	Let $H \in \mathcal{H}$. $H$ is a minor of one of $H_1+16+27+35$, $H_1+12+34+37+46$, $H_1+12+37+48+56$, $H_1+16+34+35+46$, $H_{1a}+12+18+34$, $H_{1a1}$, $H_{1b}+12+34+3x$, $H^k_{1b}+34+35+38+45+56$, $H^k_{1b1}+34+35+45+5z$, or $H^l_1+12+16+26+45$ if $H_1$ participates in the $T$-sum.		
	\end{lem}

	\begin{figure}[H]
	\input{h1.tpx}
	\end{figure}

	For each integer $r\geq 1$, let $G^r$ be the graph obtained from $G$ by adding $r$ new cubic vertices, each adjacent precisely to $\{3,5,y\}$.
	
	\begin{lem}
	Let $H \in \mathcal{H}$. $H$ is a minor of one of $H_2+12+16+26+34+35$, $H_2+12+16+26+35+37+45$, $H_{2a}+16+34+56$, $H_{2a}+34+7z$, $H_{2a}+16+7z$, $H_{2c}+12+34+3y$, $H_{2c}+12+34+3z$, $H_{2c}+12+3z+5y$, $H_{2d}+12+35+3x$, $H_{2b1}+12+34$, $H^k_{2a1}+34+35+45$, $H^k_{2a2}+34+35+45$, $H^k_{2c1}+34+35+3u+45$, $H^k_{2c1}+34+35+3y+45$, or	$H^{k,r}_{2e}+34+35+3y+45+5y$ if $H_2$ participates in the $T$-sum.		
	\end{lem}

	\begin{figure}[H]
	\input{h22small.tpx}
	\end{figure}

	\begin{lem}\label{lemma5.13}
	Let $H \in \mathcal{H}$. For each integer $m\geq 5$, $H$ is a minor of one of $ K^{3,1}_{4,3}$, $ K^{4,1}_{4,3}$, or $K^{1,3}_{m,3}$ if $K^{1,0}_{3,3}$ or $K_{4,3}$ participates in the $T$-sum.		
	\end{lem}

	Lemmas~\ref{lemma5.3}--\ref{lemma5.6} and~\ref{lemma5.10}--\ref{lemma5.13} cover all cases and yield the required results. These results combined are the desired outcome we mentioned in Lemma~\ref{lemma5.1}.

	\begin{lem}\label{lemma5.14}
	Let $H \in \mathcal{H}$. $H$ is a minor of one of $K^{3,1}_{4,3}$, $K^{4,1}_{4,3}$, $H_{2a}+16+7z$, $H_{2a}+16+34+56$, $H_{2a}+34+7z$, $H_{2c}+12+34+3y$, $H_{2c}+12+3z+5y$, $H_{1b}+12+34+3x$, $H_{2b1}+12+34$, $H_{2a1}+34+35+45$, $H_{1a1}$, $\mathrm{cube}^n+24+26+46$, $H^k_{2a2}+34+35+45$,  $H^k_{2c1}+34+35+3y+45$, $H^k_{1b1}+34+35+45+5z$, or 	$H^{k,r}_{2e}+34+35+3y+45+5y$.
	\end{lem}

	\begin{proof}
	We consider the graphs obtained by repeated $T$-sums and the following isomorphism and minor inclusions hold:
	
	$H_3+12+35+68\cong H_1+12+34+37+46\cong H_{2c}+12+34+3y$,
	
	$H^{k,j}_{3a}+34+35+37+45+57 \cong H^{k,r}_{2e}+34+35+3y+45+5y$,
	
	$H^{j}_{3b}+34+35+45+68\cong H^k_{1b}+34+35+38+45+56\cong H^k_{2c1}+34+35+3y+45$,
	
	$H_1+12+16+27\cong H_{2a}+34+7z$,
	$H_1+16+27+35\cong H_{2a}+16+7z$,
	
	$H_1+12+37+48+56\cong H_{2c}+12+3z+5y$,
	$H_1+16+34+35+46\cong H_{2a}+16+34+56$,
	
	$H_{1a}+12+18+34\cong H_{2b1}+12+34$,
	$H_{2c}+12+34+3z\preceq H_{1b}+12+34+3x$,
	
	$H_{2d}+12+35+3x\preceq H_{1b}+12+34+3x$,
	$H^k_{2c1}+34+35+3u+45\preceq H^k_{1b1}+34+35+45+5z$.
	
	$H_2+12+16+26+35+37+45\cong K^{4,1}_{4,3}$, $H_2+12+16+26+35\cong K^{3,1}_{4,3}$,
	
	$H^j_1+12+16+26+45\preceq H^k_{2c1}+34+35+3y+45$,
	
	$K^{3,1}_{4,3}\preceq K^{4,1}_{4,3}$, $K^{3',1}_{4,3}\preceq K^{4,1}_{4,3}$, $K^{1,3}_{m,3}\preceq H^{k,r}_{2e}+34+35+3y+45+5y$. 
	
	In fact, we have only demonstrated the existence of inclusion relationships. It remains to justify maximality. By Lemma~\ref{lemma3.1}, it is enough to show that every admissible edge addition or vertex split of any listed graph produces a graph containing a $W_6$ minor, unless the resulting graph is again a minor of one of the listed infinite families.
	
	The edge additions have already been checked in the preceding lemmas, so we only discuss vertex splits. For each listed finite graph, a computer-assisted check shows that every admissible split produces a graph containing a $W_6$ minor. For an infinite family, the same check reduces to its repeating core. For example, for $H^n_{2a2}+34+35+45$, a split can remain $W_6$-minor-free only, up to isomorphism, when the split of vertex $3$ satisfies $N_G(3')\setminus\{3''\}=\{4,5\}$.  In this exceptional case the resulting graph is a minor of $H^{n+1}_{2a2}+34+35+45$. The family $\mathrm{cube}^n+24+26+46$ requires one additional observation: after a vertex split of the core $\mathrm{cube}+24+26+46$, the resulting graph is either planar or is a minor of the next member of the same infinite family.
	\end{proof}

	By Lemmas~\ref{lemma5.1} and~\ref{lemma5.14}, we obtain the main conclusion of this section.
	
	\begin{lem}\label{lemma5.15}
	A $3$-connected nonplanar graph $G$ is $\{V_8,W_6\}$-minor-free if and only if $G$ is a nonplanar minor of one of $K^{3,1}_{4,3}$, $K^{4,1}_{4,3}$, $H_{2a}+16+7z$, $H_{2a}+34+7z$, $H_{2a}+16+34+56$, $H_{2c}+12+34+3y$, $H_{2c}+12+3z+5y$, $\mathrm{cube}^1+3x$, $H_{1b}+12+34+3x$, $H_{2b1}+12+34$, $H_{2a1}+34+35+45$, $H_{1a1}$, $\mathrm{cube}^n+24+26+46$, $H^k_{2a2}+34+35+45$,  $H^k_{2c1}+34+35+3y+45$, $H^k_{1b1}+34+35+45+5z$, or $H^{k,r}_{2e}+34+35+3y+45+5y$.
	\end{lem}

	\begin{proof}
	Since $C^2_6$, $DW_5$, $\mathrm{cube}$, $C^2_8$, $N$, $Q$ and $AW_8$ are planar, we only need to consider the  remaining graphs $K_5$, $K_{3,3}$, $DW^+_4$, $K_6\setminus e$, $K_6$, $\Gamma_1$, $K^{2,0}_{4,3}$, $C^2_7$,  $K^{3,0}_{4,3}$, $K^{3',0}_{4,3}$, $C^2_7+e$, $K^{2',1}_{4,3}$, $K^{3',1}_{4,3}$, $K^{4,0}_{4,3}$, $\Gamma_2$, and  $K^{4,1}_{4,3}$. 
	$K_5$ and all graphs of order 6 are minors of $K_6$, and $K_6\preceq K^{3,1}_{4,3}$. Additionally, all graphs of order 7 are minors of $K^{3,1}_{4,3}$, $K^{4,1}_{4,3}$, or $K^{1,3}_{m,3}$.
	\end{proof}

	We now prove the results stated in Section~1. Proofs are basically summaries of what we have shown.

	\begin{proof}[\textbf{Proof of Theorem~\ref{thm1.1}}]
	The result follows from Lemmas~\ref{lemma3.6} and \ref{lemma4.8}.
	\end{proof}

	\begin{proof}[\textbf{Proof of Theorem~\ref{thm1.2}}]
	The result follows from Lemmas~\ref{lemma3.5} and~\ref{lemma5.15}.
	\end{proof}

\bibliography{i4c.bib}

\newpage
\appendix
\setupappendix
\appendix

\setcounter{figure}{0}
\renewcommand{\thefigure}{A\arabic{figure}}
\setcounter{table}{0}
\renewcommand{\thetable}{A\arabic{table}}
	\section{$T$-sum of $H_1$}

For each integer $l\geq 1$, let $G^l$ be the graph obtained from $G$ by adding $l$ new cubic vertices, each adjacent precisely to $\{1,2,6\}$.

\begin{lem}
	Let $H \in \mathcal{H}$. When $H_1$ participates in the $T$-sum, $H$ can be any of the following variants: $H_1$, $H_{1a}$, $H_{1b}$, $H_{1a1}$, $H_{1b1}$, $H^k_1$, $H^l_1$, $H^k_{1b}$, and $H^k_{1b1}$, possibly with additional edges among the existing vertices, subject to the condition that the resulting graph remains $W_6$-minor-free.
\end{lem}

\begin{figure}[H]
	\input{h1.tpx}
\end{figure}

\begin{table}[H]
	\centering
	\caption{$T$-sum of $H_1$}
	\begin{tabular}{|c|c|c|c|}
		\hline
		\textbf{Graphs} & \textbf{Vertex} & \textbf{Number of Contracted Edges} & \textbf{Result} \\
		\hline
		
		\multirow{2}{*}{$H_1$ with $K_4$}
		& $5(3,4)$ & $1$ & $H_{1a}$ \\ \cline{2-4}
		& $7(6,8)$ & $1$ & $H_{1b}$ \\ \hline
		
		\multirow{2}{*}{$H_1$ with $K_{3,3}$}
		& $1(2)$   & $3$ & $H^1_1$  \\ \cline{2-4}
		& $5(3,4)$ & $3$ & $H^{1'}_1$ \\ \hline
		
		$H_{1a}$ with $K_4$
		& $6$      & $1$ & $H_{1a1}$ \\ \hline
		
		$H_{1b}$ with $K_4$
		& $6(8)$   & $1$ & $H_{1b1}$ \\ \hline
		
		$H_{1b}$ with $K_{3,3}$
		& $1(2)$   & $3$ & $H^1_{1b}$ \\ \hline
		
		$H_{1b1}$ with $K_{3,3}$
		& $1(2)$   & $3$ & $H^1_{1b1}$ \\ \hline
	\end{tabular}
\end{table}

\begin{lem}
	Let $H \in \mathcal{H}$. $H$ is a minor of one of  $H_1+16+27+35$, $H_1+12+34+37+46$,
	$H_1+12+37+48+56$, $H_1+16+34+35+46$, $H_{1a}+12+18+34$,  $H_{1a1}$, $H_{1b}+12+34+3x$, $H^l_1+12+16+26+45$,  $H^k_{1b}+34+35+38+45+56$,  or
	$H^k_{1b1}+34+35+45+5z$ if $H_1$ participates in the $T$-sum.
	
\end{lem}

\begin{proof}
	First, we obtain $W_6$-minor-free graphs by Lemma A.1.
	
	(i) $H_1+12+16+27$, $H_1+16+27+35$, $H_1+12+16+26+45$,
	$H_1+12+16+45+46$, $H_1+12+34+37+46$, $H_1+12+34+37+48$, $H_1+12+37+48+56$, $H_1+16+34+35+46$, $H_1+34+35+37+45+46+56$, $H_1+34+35+37+45+48+56$, $H^k_1+34+35+37+45+46+56$, $H^k_1+34+35+37+45+48+56$, $H^l_1+12+16+26+45$;
	
	(ii) $H_{1a}+12+18+34$;
	
	(iii) $H_{1a1}$;
	
	(iv) $H_{1b}+12+34+3x$, $H_{1b}+34+35+38+45+56$, $H_{1b}+34+35+38+45+5z$, $H^k_{1b}+34+35+38+45+56$, $H^k_{1b}+34+35+38+45+5z$;
	
	(v) $H_{1b1}+34+35+45+5z$, $H^k_{1b1}+34+35+45+5z$.
	
	Next, we examine whether there are containment relationships between different classes of graphs.
	
	$H^k_1+34+35+37+45+46+56$ \(\preceq\) 	$H^k_{1b}+34+35+38+45+56$,
	
	$H^k_1+34+35+37+45+48+56$ \(\preceq\) 	$H^k_{1b}+34+35+38+45+5z$,
	
	$H_1+12+16+26+45$ \(\preceq\) 	$H_{1a}+12+18+34$,
	
	$H_1+12+16+45+56$ \(\preceq\) 	$H_{1a}+12+18+34$,
	
	$H_1+12+34+37+48$ \(\preceq\) $H^k_{1b}+34+35+38+45+5z$,
	
	$H_1+12+16+27$ \(\preceq\)	$H^j_1+12+16+26+45$,
	
	$H^k_{1b}+34+35+38+45+5z$  \(\preceq\) $H^k_{1b1}+34+35+45+5z$.
\end{proof}

\section{$T$-sum of $H_2$}

\setcounter{figure}{0}
\renewcommand{\thefigure}{B\arabic{figure}}
\setcounter{table}{0}
\renewcommand{\thetable}{B\arabic{table}}
For integers $r,t,s,w\geq 1$, let $G^r$, $G^t$, $G^s$, and $G^w$ be the graphs obtained from $G$ by adding, respectively, $r$ new cubic vertices adjacent to $\{3,5,y\}$, $t$ new cubic vertices adjacent to $\{4,5,6\}$, $s$ new cubic vertices adjacent to $\{3,4,7\}$ and $w$ new cubic vertices adjacent to $\{1,2,7\}$.

\begin{lem}
	Let $H \in \mathcal{H}$. When $H_2$ participates in the $T$-sum, $H$ can be any of the following variants: $H_2$,
	$H_{2a}$, $H_{2a1}$, $H_{2a2}$, $H_{2b}$, $H_{2b1}$, $H_{2c}$, $H_{2c1}$, $H_{2d}$, $H_{2d1}$, $H_{2e}$, $H_{2f}$, $H_{2g}$,  $H_{2g1}$, $H^k_2$, $H^k_{2a}$, $H^t_{2a}$, $H^{k,t}_{2a}$, $H^k_{2a1}$, $H^k_{2a2}$, $H^s_{2b}$, $H^k_{2c}$, $H^k_{2c1}$, $H^k_{2d}$, $H^k_{2d1}$, $H^k_{2e}$, $H^r_{2e}$, $H^{k,r}_{2e}$, $H^w_{2f}$, $H^k_{2g}$, $H^t_{2g}$, $H^{k,t}_{2g}$, and $H^k_{2g1}$, possibly with additional edges among the existing vertices, subject to the condition that the resulting graph remains $W_6$-minor-free.
\end{lem}

\begin{figure}[H]
	\input{h22.tpx}
\end{figure}

\begin{table}[H]
	\centering
	\caption{$T$-sum of $H_{2}$}
	\begin{tabular}{|c|c|c|c|}
		\hline
		\textbf{Graphs} & \textbf{Vertex} & \textbf{Number of Contracted Edges} & \textbf{Result} \\
		\hline
		
		\multirow{3}{*}{$H_{2}$ with $K_4$}
		& $1(2)$ & $1$ & $H_{2a}$ \\ \cline{2-4}
		& $5(3)$ & $1$ & $H_{2b}$ \\ \cline{2-4}
		& $7(6)$ & $1$ & $H_{2c}$, $H_{2d}$ \\ \hline
		
		\multirow{4}{*}{$H_{2}$ with $K_{3,3}$}
		& \multirow{2}{*}{$1(2)$} & $2$ & $H_{2e}$ \\ \cline{3-4}
		& & $3$ & $H^1_{2}$ \\ \cline{2-4}
		& $5(3)$ & $3$ & $H_{2f}$ \\ \cline{2-4}
		& $7(6)$ & $3$ & $H_{2g}$ \\ \hline
		
		\multirow{2}{*}{$H_{2a}$ with $K_4$}
		& $1$ & $1$ & $H_{2a1}$ \\ \cline{2-4}
		& $z(7)$ & $1$ & $H_{2a2}$ \\ \hline
		
		\multirow{2}{*}{$H_{2a}$ with $K_{3,3}$}
		& $1$ & $3$ & $H^1_{2a}$ \\ \cline{2-4}
		& $z(7)$ & $3$ & $H^{1'}_{2a}$ \\ \hline
		
		$H^1_{2a}$ with $K_{4}$
		& $1$ & $1$ & $H^1_{2a1}$ \\ \hline
		
		$H_{2a2}$ with $K_{3,3}$
		& $1$ & $3$ & $H^1_{2a2}$ \\ \hline
		
		$H_{2b}$ with $K_4$
		& $6$ & $1$ & $H_{2b1}$ \\ \hline
		
		$H_{2b}$ with $K_{3,3}$
		& $6$ & $3$ & $H^1_{2b}$ \\ \hline
		
		\multirow{2}{*}{$H_{2c}$ with $K_4$}
		& $3$ & $1$ & $H_{2b1}$ \\ \cline{2-4}
		& $z$ & $1$ & $H_{2c1}$ \\ \hline
		
		$H_{2c}$ with $K_{3,3}$
		& $1(2)$ & $3$ & $H^1_{2c}$ \\ \hline
		
		$H_{2c1}$ with $K_{3,3}$
		& $1$ & $3$ & $H^1_{2c1}$ \\ \hline
		
		\multirow{2}{*}{$H_{2d}$ with $K_4$}
		& $1(2)$ & $1$ & $H_{2a2}$ \\ \cline{2-4}
		& $6$ & $1$ & $H_{2d1}$ \\ \hline
		
		$H_{2d}$ with $K_{3,3}$
		& $1(2)$ & $3$ & $H^1_{2d}$ \\ \hline
		
		$H_{2d1}$ with $K_{3,3}$
		& $1$ & $3$ & $H^1_{2d1}$ \\ \hline
		
		\multirow{2}{*}{$H_{2e}$ with $K_{3,3}$}
		& $1$ & $3$ & $H^1_{2e}$ \\ \cline{2-4}
		& $u(w)$ & $3$ & $H^{1'}_{2e}$ \\ \hline
		
		$H_{2f}$ with $K_{3,3}$
		& $5(x)$ & $3$ & $H^1_{2f}$ \\ \hline
		
		\multirow{3}{*}{$H_{2g}$ with $K_4$}
		& $1(2)$ & $1$ & $H^1_{2a}$ \\ \cline{2-4}
		& $3$ & $1$ & $H^1_{2b}$ \\ \cline{2-4}
		& $7(x)$ & $1$ & $H_{2g1}$ \\ \hline
		
		\multirow{3}{*}{$H_{2g}$ with $K_{3,3}$}
		& $1(2)$ & $1$ & $H^{1'}_{2g}$ \\ \cline{2-4}
		& $3$ & $1$ & $H^1_{2f}$ \\ \cline{2-4}
		& $7(x)$ & $1$ & $H^1_{2g}$ \\ \hline
		
		$H_{2g1}$ with $K_{3,3}$
		& $1$ & $3$ & $H^1_{2g1}$ \\ \hline
	\end{tabular}
\end{table}

\begin{lem}
	Let $H \in \mathcal{H}$. $H$ is a minor of one of 	 $H_2+12+16+26+34+35$, $H_2+12+16+26+35+37+45$, $H_{2a}+16+34+56$, $H_{2a}+34+7z$, $H_{2a}+16+7z$, $H_{2c}+12+34+3y$,  $H_{2c}+12+34+3z$, $H_{2c}+12+3z+5y$, $H_{2d}+12+35+3x$,
	$H^{k,r}_{2e}+34+35+3y+45+5y$, $H^K_{2a1}+34+35+45$, $H^k_{2a2}+34+35+45$, $H_{2b1}+12+34$, 	$H^k_{2c1}+34+35+3u+45$, or $H^k_{2c1}+34+35+3y+45$, 	if $H_2$ participates in the $T$-sum.
	
\end{lem}

\begin{proof}
	First, we obtain $W_6$-minor-free graphs by Lemma B.1.
	
	(i)	$H_2+12+16+26+34+35$, $H_2+34+35+37+45+56$,
	$H_2+12+16+26+35+37+45$, $H^{k}_2+34+35+37+45+56$;
	
	(ii) $H_{2a}+34+35+45+56$, $H_{2a}+16+34+56$, $H_{2a}+34+7z$, $H_{2a}+16+7z$, $H^k_{2a}+34+35+45+56$, $H^t_{2a}+34+35+45+56$, $H^{k,t}_{2a}+34+35+45+56$;
	
	(iii) $H_{2b}+12+17+34$, $H_{2b}+12+34+37$, $H_{2b}+16+17+26$, $H_{2b}+12+3y$, $H_{2b}+34+3y$, $H^s_{2b}+12+34+37$;
	
	(iv) $H_{2c}+34+35+3y+45+56$, $H_{2c}+34+35+3z+45+5y$,
	$H_{2c}+12+16+45$, $H_{2c}+12+34+3y$, $H_{2c}+12+34+3z$, $H_{2c}+12+3z+5y$, $H_{2c}+16+34+35$, $H^k_{2c}+34+35+3y+45+56$, $H^k_{2c}+34+35+3z+45+5y$;
	
	(v)	$H_{2d}+34+35+3y+45+6y$; $H_{2d}+34+35+3x+45+6y$,
	$H_{2d}+12+35+3x$, $H_{2d}+12+3x+6y$, $H_{2d}+12+45+6y$, 	$H^k_{2d}+34+35+3y+45+6y$, $H^k_{2d}+34+35+3x+45+6y$;
	
	(vi) $H_{2e}+34+35+3y+45+5y$, $H^k_{2e}+34+35+3y+45+5y$,
	$H^r_{2e}+34+35+3y+45+5y$, $H^{k,r}_{2e}+34+35+3y+45+5y$;
	
	(vii) $H_{2a1}+34+35+45$, $H^k_{2a1}+34+35+45$;
	
	(viii) $H_{2a2}+34+35+45$, $H^k_{2a2}+34+35+45$;
	
	(ix) $H_{2b1}+12+34$;
	
	(x)	$H_{2c1}+34+35+3u+45$, $H_{2c1}+34+35+3y+45$, 	$H^k_{2c1}+34+35+3u+45$, $H^k_{2c1}+34+35+3y+45$;
	
	(xi) $H_{2d1}+34+35+45+uy$, $H^k_{2d1}+34+35+45+uy$;
	
	(xii) $H_{2f}+12+17+27+34$, $H^w_{2f}+12+17+27+34$;
	
	(xiii) $H_{2g}+12+16+45+56$, $H_{2g}+34+35+45+56$, $H^k_{2g}+12+16+45+56$, $H^k_{2g}+34+35+45+56$,  $H^t_{2g}+34+35+45+56$, $H^{k,t}_{2g}+34+35+45+56$;
	
	(xiv) $H_{2g1}+34+35+45+56$, $H^k_{2g1}+34+35+45+56$.
	
	Next, we examine whether there are containment relationships between different classes of graphs.
	
	$H_2^k+34+35+37+45+56\preceq H^k_{2c}+34+35+3y+45+56$,
	
	$H_{2a}+34+7z\cong H_{2b}+34+3y$,
	
	$H_{2a}+16+7z\cong H_{2b}+12+3y$,
	
	$H_{2a}+16+34+56 \cong H_{2b}+16+17+26\cong H_{2c}+16+34+35$,
	
	$H^{k,t}_{2a}+34+35+45+56\preceq H^{k,r}_{2e}+34+35+3y+45+5y$,
	
	$H_{2b}+12+17+34\cong H_{2c}+12+16+45\preceq H_{2b1}+12+34$,
	
	$H^s_{2b}+12+34+37\preceq H^k_{2c1}+34+35+3y+45$,
	
	$H_{2b}+16+17+26\cong H_{2c}+16+34+35$,
	
	$H_{2c}+12+16+45 \preceq H_{2b1}+12+34$,
	
	$H_{2c}+12+34+3z\cong H_{2d}+12+45+6y$,
	
	$H_{2c}+12+3z+5y \cong H_{2d}+12+3x+6y$,
	
	$H^k_{2c}+34+35+3y+45+56 \cong H^k_{2d}+34+35+3y+45+6y\cong H^w_{2f}+12+17+27+34\cong H^k_{2g}+12+16+45+56\preceq H^k_{2c1}+34+35+3y+45$,
	
	$H_{2c}+34+35+3z+45+5y\cong H_{2d}+34+35+3x+45+6y\preceq H_{2c1}+34+35+3u+45$,
	
	$H_{2c1}+34+35+3u+45\cong H_{2d1}+34+35+45+uy$,
	
	$H^{k,t}_{2g}+34+35+45+56\preceq H^{k,r}_{2e}+34+35+3y+45+5y$,
	
	$H^{k}_{2g1}+34+35+45+56\cong H^{k,r}_{2e}+34+35+3y+45+5y$.
\end{proof}

\section{$T$-sum of $K^{1,0}_{3,3}$ or $K_{4,3}$}

\setcounter{figure}{0}
\renewcommand{\thefigure}{C\arabic{figure}}
\setcounter{table}{0}
\renewcommand{\thetable}{C\arabic{table}}

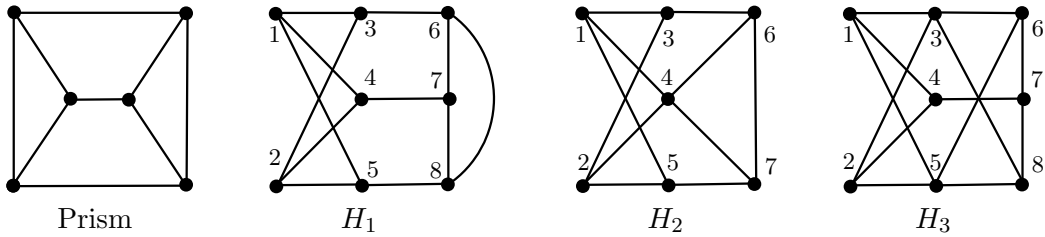
In fact, we need to consider graphs like $H_1+12$ and $H^1_1$. However, since they have already been covered in the cases for $H_1$, $H_2$, and $H_3$, we focus only on $K^{1,0}_{3,3}$ and $K^{1,0}_{4,3}$ here.

\begin{lem}
	
	Let $H \in \mathcal{H}$.
	When $K^{1,0}_{3,3}$ or $K_{4,3}$ participates in the $T$-sum, $H$ contains any of the following variants: $K^{1,0}_{3,3}$ or $K^{1,0}_{4,3}$, possibly with additional edges among the existing vertices, subject to the condition that the resulting graph remains $W_6$-minor-free.
	
\end{lem}

\begin{table}[H]
	\centering
	\caption{$T$-sum of $K^{1,0}_{3,3}$}
	\begin{tabular}{|c|c|c|c|}
		\hline
		\textbf{Graphs} & \textbf{Vertex} & \textbf{Number of Contracted Edges} & \textbf{Result} \\
		\hline
		
		\multirow{6}{*}{$K^{1,0}_{3,3}$ with $K_4$}
		& \multirow{3}{*}{$3$} & $0$ & $H_1+12$ \\ \cline{3-4}
		& & $1$ & $H_2+12$ \\ \cline{3-4}
		& & $2$ & $K^{1,1}_{3,3}$ \\ \cline{2-4}
		
		& \multirow{3}{*}{$5(4,6)$} & $0$ & $H_1+35$ \\ \cline{3-4}
		& & $1$ & $H_2+35$ \\ \cline{3-4}
		& & $2$ & $K^{2,0}_{3,3}$ \\ \hline
		
		\multirow{4}{*}{$K^{1,0}_{3,3}$ with $K_{3,3}$}
		& \multirow{2}{*}{$3$} & $2$ & $H_3+12$ \\ \cline{3-4}
		& & $3$ & $K^{1,0}_{4,3}$ \\ \cline{2-4}
		
		& \multirow{2}{*}{$5(4,6)$} & $2$ & $H_3+35$ \\ \cline{3-4}
		& & $3$ & $K^{0,1}_{4,3}$ \\ \hline
	\end{tabular}
\end{table}

\begin{table}[H]
	\centering
	\caption{$T$-sum of $K_{4,3}$}
	\begin{tabular}{|c|c|c|c|}
		\hline
		\textbf{Graphs} & \textbf{Vertex} & \textbf{Number of Contracted Edges} & \textbf{Result} \\
		\hline
		
		\multirow{3}{*}{$K_{4,3}$ with $K_4$}
		& \multirow{3}{*}{$4(1,2,3)$} & $0$ & $H^1_1$ \\ \cline{3-4}
		& & $1$ & $H^1_2$ \\ \cline{3-4}
		& & $2$ & $K^{0,1}_{4,3}$ \\ \hline
		
		\multirow{3}{*}{$K_{4,3}$ with $K_{3,3}$}
		& \multirow{3}{*}{$4(1,2,3)$} & $0,1$ & $H^1_1$ \\ \cline{3-4}
		& & $2$ & $H^1_3$ \\ \cline{3-4}
		& & $3$ & $K_{5,3}$ \\ \hline
	\end{tabular}
\end{table}

\begin{lem}
	Let $H \in \mathcal{H}$. For each integer $m\geq 5$, $H$ contains a minor of one of $K^{3,1}_{4,3}$, $K^{4,1}_{4,3}$ or $K^{1,3}_{m,3}$ if $K^{1,0}_{3,3}$ or $K_{4,3}$ participates in the $T$-sum.
	
\end{lem}

\begin{proof}
	First, we obtain $W_6$-minor-free graphs by Lemma C.1.
	
	(i) $K^{3,1}_{4,3}$, $K^{4,1}_{4,3}$, $K^{1,3}_{4,3}$;
	
	(ii) $K^{1,3}_{m,3}$.
	
	Next, we examine whether there are containment relationships between different classes of graphs.
	
	$K^{1,3}_{4,3}$ \(\preceq\)  $K^{1,3}_{m,3}$.
\end{proof}

\end{document}